\documentclass[12pt]{article}
\usepackage{amsmath, amsthm, amsfonts, url}
\usepackage{hyperref}
\newcommand{\seqnum}[1]{\href{https://oeis.org/#1}{\underline{#1}}}

\DeclareMathOperator{\rad}{rad}
\begin{document}

\theoremstyle{plain}
\newtheorem{theorem}{Theorem}
\newtheorem{proposition}[theorem]{Proposition}

\theoremstyle{definition}
\newtheorem{definition}[theorem]{Definition}
\newtheorem{conjecture}[theorem]{Conjecture}

\begin{center}
\vskip 1cm{\LARGE\bf 
The abc Conjecture Implies That
\vskip .03in
Only Finitely Many $s$-Cullen 
\vskip .1in
Numbers Are Repunits
}
\vskip 1cm
\large
Jon Grantham \\
Hester Graves \\
Institute for Defense Analyses \\
Center for Computing Sciences \\
Bowie, Maryland 20715 \\
United States \\
\href{mailto:grantham@super.org}{grantham@super.org} \\
\href{mailto:hkgrave@super.org}{hkgrave@super.org}
\end{center}

\vskip .2 in

\begin{abstract}
Assuming the abc conjecture with $\epsilon=1/6$, we use elementary methods to show that only finitely many $s$-Cullen numbers are repunits, aside from two known infinite families.  
More precisely, only finitely many positive integers $s$, $n$, $b$, and $q$ with $s,b \geq 2$ and $n,q \geq 3$ satisfy
\[C_{s,n} = ns^n + 1 = \frac{b^q -1}{b-1}.\]
\end{abstract}

\section{Introduction}

\begin{definition}  
A \textit{Cullen number} is an number of the form $C_n = n2^n +1$, where $n$ is a positive integer. The Cullen numbers are \seqnum{A002064} in the OEIS. An \textit{$s$-Cullen number}  is a number of the form $C_{n,s} = ns^n + 1$, where $s, n$ are positive integers with $s \geq 2$. See \seqnum{A050914}, for example, for the $3$-Cullen numbers.   Cullen and Dubner (\cite{Cullen} and \cite{Dubner}) introduced the two families, respectively.  
\end{definition}

The first significant result on Cullen numbers occurred in 1976, when Hooley \cite{Hooley} showed that almost all Cullen numbers are composite. 
Caldwell \cite{caldwell} conjectured that infinitely many are prime.

Luca and St\v{a}nic\v{a} \cite{lucastan} showed that the intersection of the Cullen numbers with the Fibonacci sequence is finite. 
Marques \cite{marques} generalized this result to $s$-Cullen numbers (for fixed $s$). Bilu, Marques, and Togb\'{e} \cite{bmt} (among other results) generalized
the two previous papers' results to the intersection of $s$-Cullen numbers with other recurrence sequences.

In this paper, we consider the intersection of the $s$-Cullen numbers with a two--parameter family, namely the repunits. We are able
to prove, conditionally, that the intersection of the repunits and Cullen numbers is finite, except for two infinite families. In particular, for fixed $s$, the intersection is finite.

\begin{definition} A \textit{repunit} is a positive integer $n$ that we can write as 
\[\frac{b^q -1}{b-1} = \sum_{j=0}^{q-1} b^j = b^{q-1} + b^{q-2} + \cdots + b + 1 = (1 1 \cdots 1)_{b}\]
for some integer \textit{base} $b >1$ and some  integer exponent $q$.  Beiler \cite{Beiler} introduced the name in the 1960s.
\end{definition} 

For example, base-$10$ repunits are \seqnum{A002275} in the OEIS, base-$15$ repunits are \seqnum{A135518}, and base-$2$ repunits are \seqnum{A000225}. We impose the condition $q\ge 3$ to avoid the trivial representation of any number $x$ as $11$ in base $x-1$.

\begin{definition} 
The \textit{radical} of a positive integer $n$ is the product of all the primes that divide $n$, so if $n = \prod_{p_i | n} p_i^{a_i}$, then $\rad(n) = \prod_{p_i | n} p_i$. For example, the radical of $ 90 = 2 \cdot 3^2 \cdot 5 $ is $30 = 2 \cdot 3 \cdot 5$.
\end{definition} 

\begin{conjecture}
The \textit{abc conjecture} of Oesterl\'e \cite{Oes} and Masser \cite{Masser} states that if $a$,$b$, and $c$ are relatively prime integers such that $a + b = c$, then for any $\epsilon >0$, only finitely many $(a,b,c)$ fail to satisfy the inequality
\[ c < \rad(abc)^{1+\epsilon}.\]
\end{conjecture}

In the next section, we use this conjecture with $\epsilon=1/6$.

\section{Main result}

\begin{theorem}\label{mainresult}
The abc conjecture with $\epsilon=1/6$ implies that only finitely many $s$-Cullen numbers $C_{s,n}$, with $n\ge 3$, are repunits of length three or greater.  
\end{theorem}
We divide our theorem into two cases, each of which we prove as a proposition.  The first proposition shows that only finitely many $s$-Cullen numbers can be written as repunits of length three with $n\ge 3$, and the second shows that only finitely many $s$-Cullen numbers can be written as repunits of length greater than three with $n \ge 2$.  As the union of two finite sets is finite itself, the two propositions prove the theorem.

\begin{proposition}\label{improved}  The abc conjecture with $\epsilon=1/6$ implies that only finitely many $s$-Cullen numbers are repunits of length three with $n\ge 3$.
\end{proposition}
\begin{proof}  
Suppose that $C_{s,n} = n s^n +1$ is a repunit of length three, i.e., 
$C_{s,n} = n s^n + 1 = b^2 + b +1$.  We then infer that $ns^n = b(b+1)$.  Let us first consider the case
$$b + 1  < \rad (b(b+1))^{\frac{7}{6}} = \rad(ns^n)^{\frac{7}{6}} <(ns)^{\frac{7}{6}},$$
so that
$$ns^n = b (b+1) < (b+1)^2 <(ns)^{\frac{7}{3}} = n^{\frac{7}{3}} s^{\frac{7}{3}}.$$
By taking logarithms, we see that
$$n-\frac{7}{3} < \frac{4}{3} \log_s(n),$$
or
\begin{equation}\label{first_log_contradiction}
3n-7 < 4\log_s(n).
\end{equation}

This equation cannot hold when $s \geq 9, n\geq 3$.  For each $s$, $2 \leq s \leq 8$, only finitely many values of $n$ satisfy Equation \ref{first_log_contradiction}.

The abc conjecture with $\epsilon =1/6$ asserts that only finitely many values $b$ satisfy
$b +1 \geq \rad(b(b+1))^{\frac{7}{6}}$, so certainly only finitely many $b$ satisfy $b+1 \geq \rad(b(b+1))^{\frac{7}{6}}$, where $b^2 + b+ 1$ is an $s$-Cullen number.  

We conclude that only finitely many $C_{s,n}$, with $n \geq 3$, are repunits of length three.
\end{proof}

\begin{proposition}  The abc conjecture with $\epsilon=1/6$ implies that only finitely many $s$-Cullen numbers with $n \geq  2$ are repunits of length greater than three.
\end{proposition}
\begin{proof} 
Suppose that $C_{s,n} = ns^n +1$ is a repunit of length greater than three, i.e., that $C_{s,n}$ can written as $\frac{b^q -1}{b-1} = b^{q-1} + b^{q-2} + \cdots + b +1$ for some $b \geq 2, q \geq 4$. 
Assuming this supposition,

$$ns^n = b(b^{q-2} + \cdots + b +1) = b \left ( \frac{b^{q-1} -1}{b-1} \right ),$$
which we rewrite as
$$(b-1) n s^n = b(b^{q-1} -1).$$

In the case where $b^{q-1} < \rad (b(b^{q-1}-1))^{\frac{7}{6}}$, then
$$b^{q-1} < \rad ((b-1)n s^n)^{\frac{7}{6}}$$
and thus
$$b^{q-1} -1 < (ns)^{\frac{7}{6}} (b-1)^{\frac{7}{6}},$$
which we can rewrite as
$$\frac{b(b^{q-1} -1)}{b-1} < (ns)^{\frac{7}{6}} b(b-1)^{\frac{1}{6}}.$$

The previous inequality shows us that
$$ ns^n = \frac{b(b^{q-1} -1)}{b-1}  < (nsb)^{\frac{7}{6}},$$
or
\begin{equation}
\label{eq:1}
s^{n-\frac{7}{6}} < n^{\frac{1}{6}}b^{\frac{7}{6}}.
\end{equation}

We know that 
$ns^n >b^{q-1}$, so 
$(ns^n)^{\frac{7}{6(q-1)}} >b^{\frac{7}{6}}$.  This inequality then gives us a further upper bound on Inequality \ref{eq:1}, as 
$$s^{n-\frac{7}{6}} < n^{\frac{1}{6}}b^{\frac{7}{6}}< n^{\frac{q + 6}{6(q-1)}} s^{\frac{7n}{6(q-1)}}$$
so
$$s^{n-\frac{7}{6} - \frac{7n}{6(q-1)}} < n^{\frac{q + 6}{6(q-1)}}.$$
If we take the log base $s$ of both sides, we see that 
$$n-\frac{7}{6} - \frac{7n}{6(q-1)} < \frac{q + 6}{6(q-1)}\log_s n,$$
or 
$$n - \frac{7(q-1)}{6q -13}   < \frac{q + 6}{6q-13}\log_s n.$$
We assumed that  $q \geq 4$, giving us 
$$n - \frac{21}{11} <  \frac{10}{11}\log_s n,$$
or 
\begin{equation}\label{second_log_contradiction}
11 n - 21 <  10\log_s n.
\end{equation}
We can see that this equation cannot hold if $s \geq 3$ and $n \geq 3$, or if  $s\geq 1025 $ and $n =2$.  Thus, the only potential repunits are $C_{s, 2}$, with $2 \leq s \leq 1024$.

Only finitely many $b^{q-1}$ satisfy $b^{q-1} > \rad (b(b^{q-1}-1))^{\frac{7}{6}}$, so only finitely many $s$-Cullen numbers can be written as 
$\frac{b^q -1}{b-1} = b^{q-1} + b^{q-2} + \cdots + b +1$ for said values of $b$, and thus we can conclude that only finitely many $s$-Cullen numbers with $n \geq 2$ are repunits of length greater than three.

\end{proof}

\section{Known examples} 

Two infinite families of $s$-Cullen repunits exist, but each produces at most one example for each $s$.

We see that $C_{s,1}$ is a repunit of length at least $3$ exactly when $s+1$ is. 

When $C_{s,2}$ is a repunit of length $3$, we know that $2s^2+1=x^2+x+1$ for some $s$.  Thus $s^2=x(x+1)/2$, and $s^2$ is a square triangular number. See \seqnum{A001110}. 
These numbers were characterized by Euler \cite{euler}, and help us find the infinite family $C_{6,2}$, $C_{35,2}$, $C_{204,2},\ldots$ of $s$-Cullen numbers which are length-three repunits. 

Other than the two families noted above, no $s$-Cullen numbers $C_{s,n}$  are also repunits with $s\le 100$ and $n\le 100$ or with $s\le 10^6$ and $n\le 10$. The PARI \cite{PARI} code for this computation is at \href{https://github.com/31and8191/Cullen}{github.com/31and8191/Cullen}. It is an open problem whether any $s$-Cullen repunits exist outside the two families characterized above.

\section{Acknowledgments} Both authors would like to thank the referee for the insightful and encouraging report, which led to a significantly improved Theorem \ref{mainresult}.

The second author would like to thank Ms.\ Joemese Malloy and the Thurgood Marshall Child Development Center for their extraordinary measures that allowed her to  do mathematics with peace of mind, knowing that her child was safe, loved, and well-cared for during the pandemic.  As always, she would like to thank her husband, Loren LaLonde, whose support was all the more meaningful during this difficult time.  

The first author would like to thank his wife, Christina Ruiz Grantham, for her patience during his unexpected arrival as mathematician-in-residence at the onset of the pandemic. He thanks his children, Salem and Jack, for their enthusiasm at his insertion of number theory into their elementary-school mathematics curriculum.

\bigskip
\hrule
\bigskip

\noindent 2020 {\it Mathematics Subject Classification}:
Primary 11B83.

\noindent \emph{Keywords:} Generalized Cullen Number, Cullen Number, repunit, abc conjecture.

\bigskip
\hrule
\bigskip

\noindent (Concerned with sequences
\seqnum{A002064} and
\seqnum{A002275}.)


\begin{thebibliography}{10}

\bibitem{Beiler}
Albert~H. Beiler, {\em Recreations in the {T}heory of {N}umbers: {T}he {Q}ueen
  of {M}athematics {E}ntertains}, Vol.~4, Dover Recreational Math, 2nd edition,
  2013. First edition published 1964.

\bibitem{bmt}
Yuri Bilu, Diego Marques, and Alain Togb\'{e}, Generalized {C}ullen numbers in
  linear recurrence sequences, {\em J. Number Theory} {\bf 202} (2019),
  412--425.

\bibitem{caldwell}
Chris~K. Caldwell, An amazing prime heuristic, 2000. Available at
  \url{https://www.utm.edu/staff/caldwell/preprints/Heuristics.pdf}

\bibitem{Cullen}
James Cullen, Question 15897, {\em Educ. Times}  (1905), 534.

\bibitem{Dubner}
Harvey Dubner, Generalized {C}ullen numbers, {\em J. Recreat. Math.} {\bf 21}
  (1989), 190--194.

\bibitem{euler}
Leonhard Euler, Regula facilis problemata {D}iophantea per numeros integros
  expedite resolvendi, {\em M\'emoires de l'Acad\'emie des Sciences de
  St.-P\'etersbourg} {\bf 4} (1813), 3--17. Available at
  \url{https://scholarlycommons.pacific.edu/euler-works/739/}.

\bibitem{Hooley}
C.~Hooley, {\em Applications of {Sieve} {M}ethods to the {T}heory of
  {N}umbers}, Cambridge University Press, 1976. Cambridge Tracts in Mathematics, No. 70.

\bibitem{lucastan}
Florian Luca and Pantelimon St\v{a}nic\v{a}, Cullen numbers in binary recurrent
  sequences. In {\em Applications of {F}ibonacci {N}umbers. {V}ol. 9},  Kluwer Acad. Publ., 2004, pp.
  167--175.

\bibitem{marques}
Diego Marques, On generalized {C}ullen and {W}oodall numbers that are also
  {F}ibonacci numbers, {\em J. Integer Seq.} {\bf 17} (2014), Article 14.9.4.

\bibitem{Masser}
D.~W. Masser, Abcological anecdotes, {\em Mathematika} {\bf 63} (2017),
  713--714.

\bibitem{Oes}
Joseph Oesterl\'{e}, Nouvelles approches du ``th\'{e}or\`eme'' de {F}ermat,
  {\em Ast\'{e}risque} {\bf 161-162} (1988), 165--186. S\'{e}minaire Bourbaki. Vol.\ 1987/88, Exp. No. 694.

\bibitem{PARI}
{The PARI~Group}, Univ. Bordeaux, {\em {PARI/GP {V}ersion \texttt{2.11.0}}},
  2018. Available at \url{http://pari.math.u-bordeaux.fr/}.

\end{thebibliography}
\end{document}